\documentclass[12pt]{amsart}
\usepackage{amsmath,amsthm,amssymb,amsfonts,enumerate,color}
\usepackage{mathrsfs}
\usepackage{amsmath,amstext,amsthm,amssymb,amsxtra}
\usepackage{txfonts} 
 \usepackage[colorlinks,  citecolor=blue, hypertexnames=false]{hyperref}
\allowdisplaybreaks
\usepackage{pgf,tikz}
\begin{document}
 
 \baselineskip 16pt
 \addtolength{\parskip}{8pt}


	\newtheorem{theorem}{Theorem}[section]
	\newtheorem{proposition}[theorem]{Proposition}
	\newtheorem{corollary}[theorem]{Corollary}
	\newtheorem{lemma}[theorem]{Lemma}
	\newtheorem{definition}[theorem]{Definition}	
	\newtheorem{assum}{Assumption}[section]
	\newtheorem{example}[theorem]{Example}
	\newtheorem{remark}[theorem]{Remark}
	\newtheorem*{conjecture}{Conjecture}

\newcommand{\R}{\mathbb{R}}
\newcommand{\Rn}{{\mathbb{R}^n}}
\newcommand{\T}{\mathbb{T}}
\newcommand{\Tm}{{\mathbb{T}^m}}
\newcommand{\TmRn}{{\mathbb{T}^m\times\mathbb{R}^n}}
\newcommand{\Z}{\mathbb{Z}}
\newcommand{\Zm}{{\mathbb{Z}^m}}
\newcommand{\C}{\mathbb{C}}
\newcommand{\Q}{\mathbb{Q}}
\newcommand{\N}{\mathbb{N}}

	\newcommand{\supp}{{\rm supp}{\hspace{.05cm}}}
	\newcommand {\lb}{{\langle}}
	\newcommand {\rb}{\rangle}
 	\numberwithin{equation}{section}

\title[On  smoothing estimates for    Schr\"odinger equation ]
	{On smoothing estimates for   Schr\"odinger equations
	  on  product spaces $\mathbb{T}^m\times \R^n$ }

	 \author[X.H. Chen, Z.H.  Guo, M.X.  Shen  and L.X. Yan]{Xianghong Chen, \  Zihua Guo, \ Minxing Shen\
 and \ Lixin Yan}
  \address{Xianghong Chen, Department of Mathematics, Sun Yat-sen
 University, Guangzhou, 510275, P.R. China}
 \email{chenxiangh@mail.sysu.edu.cn}
\address{Zihua Guo, School of Mathematical Sciences, Monash University, VIC 3800, Australia}
\email{zihua.guo@monash.edu}
\address{Minxing Shen, Department of Mathematics, Sun Yat-sen (Zhongshan)
University, Guangzhou, 510275, P.R. China}
\email{shenmx3@163.com}	
\address{
Lixin Yan, Department of Mathematics, Sun Yat-sen (Zhongshan) University, Guangzhou, 510275, P.R. China
}
\email{mcsylx@mail.sysu.edu.cn}

\date{\today}

	\begin{abstract}
	Let $\Delta_{\mathbb{T}^m\times \R^n} $ denote  the Laplace-Beltrami operator on the product spaces 
	$\mathbb{T}^m\times \R^n$.  In this article we 
show that
\begin{eqnarray*}
	\left\|e^{it\Delta_{\mathbb{T}^m\times \R^n}}f\right\|_{L^p(\mathbb{T}^m\times \R^n\times [0,1])}
	\leq C \|f\|_{W^{\alpha,p}(\mathbb{T}^m\times\R^n)} 
\end{eqnarray*}
	holds if   $p\geq 2(m+n+2)/(m+n)$ and $\alpha> (m+2n)(1/2-1/p)-2/p$.   
	Furthermore, 
we  apply the $\ell^2$-decoupling inequalities to establish  local $L^p$-smoothing estimates for the Schr\"odinger 
operator $e^{it\Delta_\TmRn}$ in     modulation spaces   $M_{p,q}^\alpha(\TmRn)$: 
\begin{eqnarray*}
	\|e^{it\Delta_{\TmRn}}f\|_{L^p(\TmRn\times [0,1])}\leq C \|f\|_{M_{p,q}^\alpha(\TmRn)}
\end{eqnarray*}
	for some range of $\alpha$ and $p, q$. 
The smoothing estimates in $L^p$-Sobolev and modulation spaces
    are sharp up to the endpoint regularity, in a certain range of $p$ and $q$. 
	\end{abstract}

	\maketitle

\section{Introduction}
\setcounter{equation}{0}

Consider the   Schr\"odinger equation
\begin{eqnarray} \label{e1.0}
\left\{
\begin{array}{rrr}
	i\partial_t u(x,t)+\Delta_\Rn u(x,t)&=  0, &\quad (x,t)\in\R^{n+1},\\[4pt]
u(\cdot,0)&= f. &
\end{array}
\right.
\end{eqnarray}
The solution $u$ can be formally written as
$$u(x,t)=
e^{it\Delta_{\R^n}}f(x)= \int_{\R^n} e^{i (\xi\cdot x- t|\xi|^2)}\hat{f}(\xi)\,d\xi,
$$
where $\hat{f}$ denotes the Fourier transform of $f$. For fixed $t>0$, Miyachi \cite{Mi} proved that
\begin{align}\label{Mi}
	\big\|e^{it\Delta_{\R^n}}f\big\|_{L^p(\R^n)}
	\leq C(n,p)(1+t)^{\frac\alpha2} \|f\|_{W^{\alpha,p}(\R^n)},\quad \alpha\geq  2n\Big|\frac12-\frac1p\Big|
\end{align}
holds for $1<p<\infty$,  where $W^{\alpha,p}(\R^n)$ denotes the Sobolev space. Moreover, the regularity 
exponent $\alpha$ in \eqref{Mi} is sharp. On the other hand, it is expected that for `generic' $t$, 
the exponent $\alpha$ in \eqref{Mi} can be reduced. In fact, it has been conjectured that
\begin{align}\label{RLS1}
	\|e^{it\Delta_{\R^n}}f\|_{L^p(\R^n\times [0,1])}\leq C(n,p,\alpha)\|f\|_{W^{\alpha,p}(\R^n)},\quad \alpha>  2n\Big(\frac12-\frac1p\Big)-\frac2p
\end{align}
holds for \textit{all} $p>2+ 2/n$. This is known as the \textit{local smoothing conjecture for the 
	Schr\"odinger equation}. It suggests that for `generic' $t\in[0,1]$, the solution gains $2/p$ derivative 
compared to the fixed-time estimate \eqref{Mi}.
In \cite{Ro}, Rogers showed that the Fourier restriction conjecture for the paraboloid implies the local
smoothing conjecture \eqref{RLS1}; moreover, he showed that \eqref{RLS1} holds for $p>2+ {4}/{(n+1)}$ 
by using the Fourier restriction theorem of Tao \cite{Ta}. 
 We refer to \cite{GaMiaZh} for recent progress 
 on the local smoothing conjecture \eqref{RLS1}. 
For background on local smoothing type estimates, we refer 
to \cite{So} (see also \cite{GuWaZha} for recent progress on the corresponding problem for the wave equation).

In this article,  we are interested the analog of the inequality \eqref{RLS1} in the semiperiodic setting 
 ${\mathbb T^m}\times {\mathbb R^n}, m,n\geq 1,$ i.e., 
\begin{eqnarray} \label{e1.111}	
	i\partial_t u(x, y, t)+\Delta_\TmRn u(x, y, t)
	= 0
\end{eqnarray}
with the initial data
\begin{eqnarray} \label{e1.211}
u(x, y, 0)=f(x,y).
\end{eqnarray}
In this case, 
  $(x,y, t)\in \TmRn \times {\mathbb R}$, and 
$\Delta_\TmRn = \Delta_x +\Delta_y$ with   $\Delta_x$ is the Laplace-Beltrami operator on 
$\mathbb{T}^m$ and  $\Delta_y=\sum_{j=1}^n\partial^2_{y_j}$ is the Laplace  operator on 
$\mathbb{R}^n$.  The solution of \eqref{e1.111} and \eqref{e1.211}   is given by
\begin{eqnarray}\label{e1.1111}
	u(x,y,t) &=&e^{it\Delta_\TmRn}f(x,y) 
	=\sum_{k\in\mathbb{Z}^m} 
	\int_{\R^n} e^{i(k\cdot x +\xi\cdot y)}e^{ -it(   |k|^2+ |\xi|^2)} \hat{f}(k,\xi)\,d\xi.
\end{eqnarray}
Let $2\leq p\leq \infty$ and set $ s(p, m)=m\big({1/2}- {1/p}\big)$. From  the imbeddings $W^{s(p, m), p}(\mathbb T^m) \hookrightarrow W^{s(p, m), 2}(\mathbb T^m) 
\hookrightarrow  L^p(\mathbb T^m)$ and
the identity   $  \|   e^{it \Delta_{\mathbb T^m}}g \|_{ W^{s(p, m), 2}(\mathbb T^m)  }
=\|   g \|_{ W^{s(p, m), 2}(\mathbb T^m)  }$, which is clear since  $e^{it \Delta_{\mathbb T^m}}$  is an
isometry   on   $L^2(\mathbb T^m)$, we have the estimate   for $2\leq  p<\infty$,   
\begin{align}\label{temp}
	\big\|e^{it\Delta_{\mathbb{T}^m}}g\big\|_{L^p(\mathbb{T}^m)}\leq C(m,p)  \|g\|_{W^{s(p, m),p}(\mathbb{T}^m)}.
\end{align}
By writing
\begin{eqnarray}
 \label{commute}
	e^{it\Delta_\TmRn}f(x,y)= e^{it\Delta_{\R^n}}(e^{it\Delta_{\mathbb{T}^m}}f(x,\cdot))(y),
\end{eqnarray}
we then apply \eqref{temp} and  \eqref{Mi} to see that 
\begin{align}\label{fixtime}
	\big\|e^{it\Delta_\TmRn}f\big\|_{L^p(\mathbb{T}^m\times \R^n)} \leq C(m,n,p,t,\alpha) \|f\|_{W^{\alpha,p}(\mathbb{T}^m\times \R^n)},\quad
	\alpha\geq (m+2n)\Big|\frac12-\frac1p\Big|.
\end{align} 
It turns out that for `generic' $t\in[0,1]$, the solution also gains $2/p$ derivative 
  compared to the fixed-time estimate \eqref{fixtime}, at least for certain range of $p$. Indeed, by the 
  Strichartz estimates on $\Tm$ (\cite{Bo93, BoDe, KV}) and H\"older's inequality, one has
  that   for $p>2+ 4/m$,
  \begin{align}\label{BDintro}
  	\big\|e^{it\Delta_{\Tm}}f\big\|_{L^p(\Tm\times [0,1])}\leq C(m,p)  \|f\|_{W^{s, p}(\Tm)}, \ \ \ 
  	s \geq  m\Big(\frac12-\frac1p\Big)-\frac2p.
  \end{align}
  Combining \eqref{Mi} and \eqref{BDintro},  we have that  for $p>2+ 4/m$,
  \begin{align*}
  	\big\|e^{it\Delta_\TmRn}f\big\|_{L^p(\mathbb{T}^m\times \R^n\times [0,1])}
  	&\leq C(m,n,p)  \|f\|_{W^{\alpha,p}(\mathbb{T}^m\times \R^n)}, \ \ \ \alpha \geq (m+2n)\Big(\frac12-\frac1p\Big)-\frac2p.
  \end{align*}

The first aim of this article is to   extend  this range to $p\geq 2+ 4/({m+n})$. More precisely, we have the following result.

\begin{theorem}\label{th1.1}
	Let $p\geq 2+4/({m+n})$  and let $\alpha> (m+2n)(1/2-1/p)-2/p$. Then  there is a constant $C=C(m,n,p,\alpha)$ such   
\begin{eqnarray}\label{e1.2}
	\big\|e^{it\Delta_\TmRn}f\big\|_{L^p(\mathbb{T}^m\times \R^n\times [0,1])} \leq C \|f\|_{W^{\alpha,p}(\mathbb{T}^m\times\R^n)}.
\end{eqnarray}
Moreover, \eqref{e1.2} fails when $0\le\alpha< (m+2n)(1/2-1/p)-2/p$.
\end{theorem}

 To prove Theorem~\ref{th1.1}, we combine ideas of Rogers \cite{Ro} 
and sharp Strichartz-type  estimates for solutions to the Schr\"odinger equation \eqref{e1.111} on $\TmRn$ due to Barron  \cite{B}. 
Note that $(m+2n)(1/2-1/p)-2/p=0$ when $p=2+4/({m+2n})$. 
It may be possible to extend the estimate \eqref{e1.2} 
for the  range to $p> 2+ 4/(m+2n)$, however, it is not clear for us yet.

The second aim of this article is to apply  the $\ell^2$-decoupling inequalities of Bourgain and Demeter \cite{BoDe} to obtain the local $L^p$-smoothing estimates for the Schr\"odinger equation   by working with a different space of initial data in modulation spaces,  which 
are compared to  initial data in Sobolev space  in \eqref{e1.2}.     
Recall that modulation spaces was introduced by  Feichtinger \cite{F1}. 
  The local  $L^p$-smoothing estimates  for the Schr\"odinger equation \eqref{e1.0} in modulation spaces  were  first discussed by Schippa    in  \cite{Sc}.   In the last decades, modulation spaces in the context of Fourier multipliers and Schr\"odinger equations have been extensively studied, see for examples \cite{BGOR, BC, Sc, WHHG} and the references therein.

 For the definition of modulation spaces,   consider the Fourier multipliers
$$
\widehat{\square_K f}= \sigma_K \hat f, \quad K\in\Z^{m+n},
$$
with $\{\sigma_K\}_{K\in\Z^{m+n}}\subseteq C_c^{\infty}({\mathbb R^{m+n}})$ a smooth partition of unity, adapted to the translated unit cubes $Q_K=K+[-{1\over 2}, {1\over 2})^{m+n}$.  Without loss of generality, we assume that  $\rm{supp}\, \sigma_0\subset [-{3\over 4}, {3\over 4})$$^{m+n}$,     $\sigma_K(\cdot)= \sigma_0(\cdot-K)$, and so 
$$
	\sigma_K f(x,y)= \int_{|\xi-k^\prime|\leq 1} \sigma(k,\xi)\hat f(k,\xi) e^{i(k\cdot x+\xi \cdot y)}\,d\xi,
$$
where $K=(k,k^\prime)\in\Zm\times\Z^n$.
Let $0<p,q<\infty$ and $s\geq 0$.
Write $\langle K \rangle= (1+|K|^2)^{1/2}$. The norm is defined by
$$
\|f\|_{M_{p,q}^\alpha(\TmRn)}= \bigg( \sum_{K\in\Z^{m+n}} \langle K \rangle^{q\alpha}\|\square_K f\|_{L^p(\TmRn)}^q \bigg)^{1/q}.
$$
 If $\alpha=0$, we write $M_{p,q} (\TmRn)$.
Modulation spaces are closely related with $L^p$ spaces.
By Plancherel's theorem, we have that 
  $M_{2,2} (\TmRn)= L^2  (\TmRn)$.
  By the embedding of $\ell^p$-spaces and Bernstein's inequality,
\begin{eqnarray}
	\label{e1.12}
 M_{p,q_1}^\alpha(\TmRn)  \hookrightarrow M_{p,q_2}^\alpha(\TmRn)  \quad (q_1\leq q_2),
\end{eqnarray}
\begin{eqnarray}
	\label{e1.13}
 M_{p_1,q}^\alpha(\TmRn)  \hookrightarrow M_{p_2,q}^\alpha(\TmRn)  \quad (p_1\leq p_2).
\end{eqnarray}
Rubio de Francia's inequality and duality yield 
\begin{eqnarray}
	\label{e1.14}
M_{p,p'} (\TmRn)  \hookrightarrow  L^p (\TmRn) \hookrightarrow M_{p, p} (\TmRn) \quad (2\leq p\leq \infty ),
\end{eqnarray}
\begin{eqnarray}
	\label{e1.15}
M_{p,p} (\TmRn)  \hookrightarrow  L^p(\TmRn)  \hookrightarrow M_{p, p'} (\TmRn) \quad (1\leq p\leq 2 ). 
\end{eqnarray}
By H\"older's inequality,
\begin{eqnarray}
	\label{e1.16}
M_{p,q_1}^{\alpha_1}(\TmRn)\hookrightarrow M_{p,q_2}^{\alpha_2}(\TmRn)
\end{eqnarray}
whenever $ \alpha_1-\alpha_2> (m+n)(1/q_2-1/q_1)>0$.

Our result is the following.
 
\begin{theorem}\label{th1.2}
	Let $p\geq 2$ and $1\leq q<\infty$. Then  there is a constant $C=C(m,n,p,q,\alpha)$ such that
	\begin{align}\label{LSM}
		\|e^{it\Delta_{\TmRn}}f\|_{L^p(\TmRn\times [0,1])}\leq C \|f\|_{M_{p,q}^\alpha(\TmRn)}
	\end{align}
provided that  $\alpha>\alpha(p,q)$, where 		
	\begin{itemize}
		\item[(a)] $\alpha(p,q)=\max \{ 0, (m+n)(1/2-1/q)\}$  with  $2\leq p\leq 2+4/(m+n)$.

		\item[(b)]  $\alpha(p,q)=(m+n)(1-1/p-1/q)-2/p$ with $p\geq 2+4/(m+n)$ and  $q\geq 2$.
		
		\item[(c)] $\alpha(p,q)=2(1-1/q)[(m+n)(1/2-1/p)-2/p]$ with  $p\geq 2+4/(m+n)$ and $1\leq q\leq 2$.
	\end{itemize}
\end{theorem}

The sharpness of the range 
 $\alpha$ (up to endpoints) in Theorems \ref{th1.1} and \ref{th1.2} is a consequence
of the following   proposition. It is seen   that
Theorem~\ref{th1.2} is sharp up to the endpoint regularity for 
$2\leq  p\leq 2+4/(m+n)$ and $1\leq q\leq 2$,  and for $ 2+4/(m+n)\leq p\leq \infty$ and $2\leq q\leq \infty$.

\begin{proposition}\label{prop1.3}
	 (i)	Let $p, q, r\ge 1$.
		Suppose
		\begin{align}\label{eq:LS}
			\big\| e^{it  \Delta_{\mathbb T^m\times\mathbb R^n}} f \big\|_{L^q(\mathbb{T}^m\times \R^n, L^r[0,1])}
			\leq C  \| f\|_{W^{\alpha, p}(\mathbb{T}^m\times \R^n)}
		\end{align}
		holds for some $\alpha\ge 0$. Then
		$$
		\alpha
		\ge
		(m+n)\Big(\frac12-\frac1q\Big)
		+n\Big(\frac 12-\frac1p\Big)
		-\frac2r.  
		$$
		In particular, when $p=q=r$, one has $ 
		\alpha\geq (m+2n)\big(1/2- 1/p\big)-2/p;
		$ 
		when $p=2$, $q=r$, one has $\alpha\geq(m+n)\big(1/2-1/q\big)- 2/q$.
 
	(ii) Let $p, q\ge 1$. Suppose
	\begin{align}\label{eq:modulationn}
		\big\| e^{it  \Delta_{\mathbb T^m\times\mathbb R^n}} f \big\|_{L^p(\mathbb{T}^m\times \R^n\times [0,1])}
		\leq C  \|f\|_{M^\alpha_{p,q}(\mathbb{T}^m\times \R^n)}
	\end{align}
	holds for some $\alpha$. Then 
	$$
	\alpha\geq \max\Big\{0,(m+n)\big(1-\frac1p-\frac1q\big)-\frac2p\Big\}.
	$$
\end{proposition}

The paper is organized as follows. In Section \ref{sec-pre}, we give some preliminaries on the Littlewood-Paley theory on $\TmRn$.
In Section  \ref{sec-1.1}, we prove Theorem \ref{th1.1} by making use of 
Strichartz-type  estimates of Barron \cite{B} for solutions to the Schr\"odinger equation \eqref{e1.111} on $\TmRn$.
 In Section \ref{sec-4},
we apply the $\ell^2$ decoupling theorem of Bourgain and Demeter \cite{BoDe} to  show Theorem \ref{th1.2}.
In Section \ref{sec-nec}, we
show the sharpness of Theorems \ref{th1.1} and \ref{th1.2} by proving a necessary condition for more general space-time estimates.

\medskip

\noindent{\bf Notation.}
We will write $A\leq CB$ when $C$ is a constant depending only on $m,n$, $p,q,r, \alpha$ and $s$; 
we write $A\leq C_\varepsilon B$ (resp. $A\leq C_N B$) to indicate that $C$ may depend additionally 
on $\varepsilon$ (resp. $N$). The values of $C$, $C_\varepsilon$ and $C_N$ may change from line to line.


\section{Preliminaries}\label{sec-pre}
\setcounter{equation}{0}

Throughout the article, we identify $\Tm$ with the periodic cube $[0,2\pi]^m\subset\R^m$. 
For functions $f(x,y)$ defined on $\TmRn$, we use the following convention for the Fourier transform:
$$\hat f(k,\xi)
=\frac{1}{(2\pi)^{m+n}}\int_{\TmRn} f(x,y) e^{-i(k\cdot  x+\xi\cdot y)} dxdy
,\quad k\in\Zm,\,\xi\in\Rn.$$
Consequently, for sufficiently good functions $f$, we have
$$f(x,y)=\sum_{k\in\Zm}\int_{\Rn} \hat f(k,\xi) e^{i(k\cdot x+\xi \cdot y)}\,d\xi.$$
A function $m(k,\xi)$ on $\Zm\times\Rn$ is said to be a multiplier on $L^p(\TmRn)$ if
$$Tf
=\sum_{k\in\Zm}\int_{\Rn} m(k,\xi)\hat f(k,\xi) e^{i(k\cdot x+\xi\cdot y)}\,d\xi$$
defines a bounded operator on $L^p(\TmRn)$ (the operator norm of $T$ will be denoted by $\|m\|_{\mathcal{M}^p(\TmRn)}$).
When $m(k,\xi)=m(-|k|^2-|\xi|^2)$ for some one-dimensional function $m(\cdot)$, we denote corresponding operator $T$ by $m(\Delta_{\mathbb T^m\times\mathbb R^n})$. In particular,
\begin{equation}\label{fourier-inv}
e^{it\Delta_{\mathbb T^m\times\mathbb R^n}}f(x,y)=\sum_{k\in\Zm}\int_{\Rn} \hat f(k,\xi) e^{i(k\cdot x+\xi\cdot  y)}e^{-it(|k|^2+|\xi|^2)}\,d\xi.
\end{equation}
For $\alpha\ge 0$, the Sobolev norm is defined by
$$\|f\|_{W^{\alpha,p}(\TmRn)}=\big\|(1-\Delta_{\mathbb T^m\times\mathbb R^n})^{\frac\alpha2}f\big\|_{L^p(\TmRn)}.$$

Below we outline some basic facts related to the Littlewood-Paley inequality for Sobolev functions. 
The first is a simple analogue of the classical transference theorem on $\Tm$ (see for example, \cite[Section 4.3.2]{Gr} and \cite{Sh}).

\begin{proposition}\label{AP1v2}
	Let $1\le p\le \infty$. Suppose $m(\eta,\xi)$ is a continuous multiplier on $L^p(\R^{m+n})$. 
	Then $m|_{\Zm\times\Rn}(k,\xi)$ is a multiplier on  $L^p(\mathbb{T}^m\times \R^n)$, with
	$\|m\|_{\mathcal{M}^p(\TmRn)}\le \|m\|_{\mathcal{M}^p(\R^{m+n})}$.
\end{proposition}

The Littlewood-Paley inequality on $\TmRn$ follows from Proposition \ref{AP1v2} and the classical 
Littlewood-Paley inequality on $\R^{m+n}$.  Let $\psi$ be a Schwartz function on $\R^{m+n}$ with
$$
\text{supp}\widehat{\psi}\subset\{(\eta,\xi)\in\R^m\times\Rn:  1/2\leq \sqrt{|\eta|^2+|\xi|^2} \leq 2\}.
$$
Let $\widehat{\psi}_j(\eta,\xi)= \widehat{\psi}(\eta/2^j,\xi/2^j)$, $j\ge 1$, and $\psi_0$ be a Schwartz function on $\R^{m+n}$ with its Fourier transform supported in the unit ball.
For functions $f$ on $\mathbb{T}^m\times \R^n$, set $\widehat{P_jf}(k,\xi)= \widehat{\psi}_j(k,\xi)\hat{f}(k,\xi)$.

\begin{proposition}[Littlewood-Paley inequality]\label{AC2v2}
	Let $1<p<\infty$. Then for any $\alpha\ge 0$, we have
$$
	\Big\| \Big(\sum_{j\ge 0} |2^{j\alpha}P_jf|^2 \Big)^{1/2}\Big\|_{L^p(\mathbb{T}^m\times \R^n)}
	\leq C \|f\|_{W^{\alpha,p}(\mathbb{T}^m\times \R^n)}.
$$
Moreover, assuming $\sum_{j\geq 0}\psi_j=1$, we have the reverse inequality
$$
	\|f\|_{W^{\alpha,p}(\mathbb{T}^m\times \R^n)}\leq C\Big\| \Big(\sum_{j\ge 0} |2^{j\alpha}P_jf|^2 \Big)^{1/2}\Big\|_{L^p(\mathbb{T}^m\times \R^n)}.
$$
\end{proposition}

The following corollaries of Proposition \ref{AC2v2} will be used in the proof of Theorems \ref{th1.1} and \ref{th1.2}.

\begin{corollary}\label{cor}
(i) Let $1<p<\infty$ and let $\alpha\ge 0$. Then for any function $f$ on $\TmRn$ with $\text{supp}\hat f\subset B_R^{m+n}(0)$, $R\ge 1$, we have
$$\|f\|_{W^{\alpha,p}(\mathbb{T}^m\times \R^n)}
\le C R^\alpha \|f\|_{L^{p}(\mathbb{T}^m\times \R^n)}.$$

\noindent(ii) Let $2\le p<\infty$ and let $X$ be a Banach space. Suppose a linear operator $T$ satisfies
$$\|Tf\|_X
\le C R^\alpha \|f\|_{L^{p}(\mathbb{T}^m\times \R^n)}$$
for all functions $f$ on $\TmRn$ with $\text{supp}\hat f\subset \big(\mathbb Z^m\cap[-R,R]^m\big)\times\mathbb [-R,R]^n$, $R\ge 1$. Then for any $\varepsilon>0$, we have
$$\|Tf\|_X
\le C_\varepsilon  \|f\|_{W^{\alpha+\varepsilon,p}(\mathbb{T}^m\times \R^n)}.$$
\end{corollary}


\medskip


\medskip

\section{Proof of Theorem~\ref{th1.1}}\label{sec-1.1}
\setcounter{equation}{0}
 
   To prove Theorem \ref{th1.1}, as a starting point one  needs the following Proposition~\ref{prop3.1} on sharp Strichartz-type  estimates for solutions to the Schr\"odinger equation \eqref{e1.111} on $\TmRn$, which was proved by Barron    \cite[Proposition  3.4]{B}.

 \begin{proposition}\label{prop3.1}
 	Let $p\geq 2+ 4/({m+n})$ and let $\alpha> (m+n)(1/2-1/p)- 2/p$. Then we have
 	\begin{eqnarray}\label{e1.3}
 		\big\|e^{it\Delta_\TmRn}f\big\|_{L^p(\mathbb{T}^m\times \R^n\times [0,1])}\leq C(m,n,p,\alpha) \|f\|_{W^{\alpha,2}(\mathbb{T}^m\times\R^n)}.
 	\end{eqnarray}
 \end{proposition}
  
  In  the breakthrough paper \cite{BoDe}, Bourgain and Demeter proved sharp Strichartz estimates on $\mathbb T^n$ 
  by establishing a sharp $\ell^2$ decoupling theorem in $\mathbb R^{n+1}$. Their approach makes use of periodicity 
  and parabolic rescaling to pass between the periodic and Euclidean settings. In order to
  obtain  Proposition \ref{prop3.1}, 
  Barron used the same ideas to reduce the problem to the setting of $\mathbb R^{m+n+1}$. The needed $\ell^2$ decoupling
  is now carried out with respect to the $(m+n)$-dimensional paraboloid in $\mathbb R^{m+n+1}$, at a scale adapted to the
  periodic component $\mathbb T^m$. The decoupled pieces are initially measured at a subcritical integrability exponent
  for both $\mathbb T^m$ and $\mathbb R^n$. 
  However, these pieces are essentially constant in the periodic direction 
  and band-limited in the nonperiodic direction. This observation, together with a
  Strichartz inequality with mixed norm, allows   to bound each piece efficiently to get the desired estimate.

 A standard scaling argument shows that Proposition~\ref{prop3.1} is sharp in that \eqref{e1.3} fails when
  $0\le\alpha< (m+n)\left(\frac12-\frac1p\right)-\frac2p$; see also Proposition \ref{prop1.3}.
  When $m=n=1$, the endpoint version of \eqref{e1.3} takes the form
  \begin{equation}\label{TT}
  	\big\|e^{it\Delta_\TmRn}f\big\|_{L^4(\mathbb{T}\times \R\times [0,1])} \leq C \|f\|_{L^{2}(\mathbb{T}\times\R)},
  \end{equation}
  and has been established earlier by Takaoka and Tzvetkov \cite{TaTz}.  Furthermore,    Barron  \cite{B} show that the $\epsilon$-loss
  from Strichartz estimates \eqref{e1.3}  can be remove for $p>2+ 4/(m+n)$. That is  for $p> 2+ 4/({m+n})$, 
$$ 
\big\|e^{it\Delta_\TmRn}f\big\|_{L^p(\mathbb{T}^m\times \R^n\times [0,1])}\leq C(m,n,p,\alpha) \|f\|_{W^{\alpha,2}(\mathbb{T}^m\times\R^n)}
$$ 
with    $\alpha=(m+n)(1/2-1/p)- 2/p$.
 There is a large body of literature related to the Strichartz estimates for Schr\"odinger equation, we refer the reader to \cite{B, BCP, CGYZ, GOW, IP, KV, TaTz} 
  and the references therein.

Next we deduce Theorem \ref{th1.1} from Proposition~\ref{prop3.1} using
a localization argument for the Schr\"odinger equation.
Such an argument is relatively standard in the Euclidean setting (cf. \cite{Ta}, \cite{Ro} and references therein). It allows one
to consider only initial data $f$ that are (essentially) localized at a scale proportional to its frequency, say $R$. 
In the cylindrical setting $\mathbb T^m\times\mathbb R^n$, we adapt a localization argument from \cite[Lemma 8]{Ro} to 
show that, in order to prove Theorem \ref{th1.1}, it suffices to consider functions $f$ that are essentially supported in 
balls of the form $\mathbb T^m\times B^n_{R^{1+\varepsilon}}$. This effectively allows us to apply H\"older's inequality to
 deduce the local smoothing estimate \eqref{e1.2} from the Strichartz estimate \eqref{e1.3}.

Let us first summarize the Euclidean estimates we need in the proof. These estimates are direct corollaries of  \cite[Lemma 7]{Ro} by rescaling.

\begin{lemma}[\cite{Ro}, Lemma 7]\label{localizaiton1}
Let $\varepsilon>0$ and let $R\ge 1$. There exist Schwartz functions $\eta_\ell\in\mathscr S(\mathbb R^n)$, 
$\ell\in\mathbb Z^n$, with $\text{supp}\widehat {\eta_\ell}\subset \mathbb [-R^{-1},R^{-1}]^n$,
such that the following statements hold. \\
(i) For any function $f$,
\begin{equation}\label{local-holder}
\bigg(\sum_{\ell\in\mathbb Z^n} \| f\eta_\ell\|_{L^q(\mathbb R^n)}^p
\bigg)^{1/p}
\le C_\varepsilon R^{
n\left(\frac1q-\frac1p\right)+\varepsilon}
\|f\|_{L^p(\mathbb R^n)},\quad 1\le q\le p\le\infty;
\end{equation}
(ii) For any function $f$ with $\text{supp}\hat f\subset \mathbb [-R,R]^n$, any $t\in[0,1]$, and any $N\ge 1$,
\begin{equation}\label{local-pointwise}
\big\|e^{it\Delta_{\mathbb R^n}}f\big\|_{L^p(\mathbb R^n)}
\le
\bigg(\sum_{\ell\in\mathbb Z^n} \big\| e^{it\Delta_{\mathbb R^n}}(f\eta_\ell)\big\|_{L^p(\mathbb R^n)}^p
\bigg)^{1/p}
+ C_N R^{-N}\|f\|_{L^p(\mathbb R^n)},\quad
1\le p\le \infty.
\end{equation}
\end{lemma}

The next lemma is a cylindrical analogue of \cite[Lemma 8]{Ro} for $\mathbb R^n$. The proof makes use of 
the product structure of $\mathbb T^m\times\mathbb R^n$ so that Lemma \ref{localizaiton1} can be applied 
directly to the Euclidean component. We remark that the same proof works for manifolds the form $M\times\mathbb R^n$ 
(where $M$ is any compact Riemannian manifold). For our application we will only be concerned with the case $M=\mathbb T^m$ here.

\begin{lemma}\label{localizaiton2}
Let $p>2$ and let $s\ge 0$. Suppose
\begin{equation}\label{strichartz-assumed}
\big\|e^{it\Delta_{\mathbb T^m\times\mathbb R^n}} f\big\|_{L^p(\mathbb T^m\times\mathbb R^n\times [0,1])}
\le C R^s\|f\|_{L^2(\mathbb T^m\times\mathbb R^n)}
\end{equation}
holds for functions $f$ with $\text{supp}\hat f\subset \big(\mathbb Z^m\cap[-R,R]^m\big)\times\mathbb [-R,R]^n$, $R\ge 1$. 
Then, for the same class of functions  $f$, we have
\begin{equation}\label{strichartz-LS}
\big\|e^{it\Delta_{\mathbb T^m\times\mathbb R^n}} f\big\|_{L^p(\mathbb T^m\times\mathbb R^n\times [0,1])}
\le C_\varepsilon R^{s+n\left(\frac12-\frac1p\right)+\varepsilon}
\|f\|_{L^p(\mathbb T^m\times\mathbb R^n)},
\quad \varepsilon>0.
\end{equation}
\end{lemma}

\begin{proof}
Fix $x\in\mathbb T^m$ and $t\in[0,1]$. Recall from \eqref{commute} that
$$
(e^{it\Delta_{\mathbb T^m\times\mathbb R^n}} f)(x,y)
=e^{it\Delta_{\mathbb R^n}}
\big((e^{it\Delta_{\mathbb T^m}} f)(x,\cdot)\big)(y),\quad (x,y)\in \mathbb T^m\times\mathbb R^n.
$$
Applying \eqref{local-pointwise} to $(e^{it\Delta_{\mathbb T^m}} f)(x,\cdot)$, we obtain for any $N\ge 1$,
\begin{align*}
\big\|(e^{it\Delta_{\mathbb T^m\times\mathbb R^n}} f)(x,\cdot)\big\|_{L^p(\mathbb R^n)}
&\le \bigg(\sum_{\ell\in\mathbb Z^n} \big\| (e^{it\Delta_{\mathbb T^m\times\mathbb R^n}} (f\eta_\ell))(x,\cdot)\big\|_{L^p(\mathbb R^n)}^p
\bigg)^{1/p}\\
&\quad+ C_N R^{-N}\big\|(e^{it\Delta_{\mathbb T^m}} f)(x,\cdot)\big\|_{L^p(\mathbb R^n)},
\end{align*}
where we have used the identity $(e^{it\Delta_{\mathbb T^m}} f)\eta_\ell(y)=e^{it\Delta_{\mathbb T^m}} (f\eta_\ell(y))$ in the first term.
Integrating over $x$ and $t$, we then have
\begin{align*}
\big\|e^{it\Delta_{\mathbb T^m\times\mathbb R^n}} f\big\|_{L^p(\mathbb T^m\times\mathbb R^n\times[0,1])}
&\le \bigg(\sum_{\ell\in\mathbb Z^n} \big\| e^{it\Delta_{\mathbb T^m\times\mathbb R^n}} (f\eta_\ell)\big\|_{L^p(\mathbb T^m\times\mathbb R^n\times[0,1])}^p
\bigg)^{1/p}\\
&\quad+ C_N R^{-N}\big\|e^{it\Delta_{\mathbb T^m}} f\big\|_{L^p(\mathbb T^m\times\mathbb R^n\times[0,1])}\\
&=: I + II.
\end{align*}
Applying the assumption \eqref{strichartz-assumed} to the first term \textit{I}, we see that
$$
I \le
C R^s \bigg(\sum_{\ell\in\mathbb Z^n}
\| f\eta_\ell\|_{L^2(\mathbb T^m\times\mathbb R^n)}^p
\bigg)^{1/p}.
$$
By H\"older's inequality,
$$\| f\eta_\ell\|_{L^2(\mathbb T^m\times\mathbb R^n)}^p
\le C^p\int_{\mathbb T^m}\| f(x,\cdot)\eta_\ell\|_{L^2(\mathbb R^n)}^p\,dx.$$
Thus we can bound
$$
I \le
C R^s \Bigg(
\int_{\mathbb T^m}
\sum_{\ell\in\mathbb Z^n} \| f(x,\cdot)\eta_\ell\|_{L^2(\mathbb R^n)}^p\,dx
\Bigg)^{1/p}.
$$
Applying \eqref{local-holder} to the sum with $q=2$, we get
\begin{align}
I
&\le
C_\varepsilon R^{s+
n\left(\frac1q-\frac1p\right)+\varepsilon}
\left(
\int_{\mathbb T^m}
\| f(x,\cdot)\|_{L^p(\mathbb R^n)}^p\,dx
\right)^{1/p}\notag\\
&= C_\varepsilon R^{s+
n\left(\frac1q-\frac1p\right)+\varepsilon}
\|f\|_{L^p(\mathbb T^m\times\mathbb R^n)}.\label{term-I}
\end{align}

For the second term \textit{II}, notice that for any fixed $y$ and $t$, by simple $L^\infty$ estimate on $\mathbb T^m$,
$$\big\|e^{it\Delta_{\mathbb T^m}} f(\cdot,y)\big\|_{L^p(\mathbb T^m)}
\le C R^{\frac m2}\|f(\cdot,y)\|_{L^p(\mathbb T^m)}.
$$
Therefore, after integrating over $y$ and $t$, we have
\begin{equation}\label{term-II}
II\le C_N R^{-N+\frac m2} \|f\|_{L^p(\mathbb T^m\times\mathbb R^n)}.
\end{equation}
Now choosing $N\ge \frac m2$ in \eqref{term-II}, and combining \eqref{term-I} and \eqref{term-II},
 we obtain the desired estimate \eqref{strichartz-LS}. This completes the proof of Proposition \ref{localizaiton1}.
\end{proof}

Theorem \ref{th1.1} now follows from Proposition~\ref{prop3.1}, Lemma \ref{localizaiton2}, and the Littlewood-Paley inequality.

\begin{proof}[Proof of Theorem \ref{th1.1}]
By Corollary \ref{cor}(ii), it suffices to show that for  $p\ge\frac{2(m+n+2)}{m+n}$ and any $\varepsilon>0$,
$$
\big\|e^{it\Delta_{\mathbb T^m\times\mathbb R^n}} f\big\|_{L^p(\mathbb T^m\times\mathbb R^n\times [0,1])}
\le C_\varepsilon R^{(m+2n)\left(\frac12-\frac1p\right)-\frac2p+2\varepsilon}
\|f\|_{L^p(\mathbb T^m\times\mathbb R^n)}
$$
holds for functions $f$ with $\text{supp}\hat f\subset \big(\mathbb Z^m\cap[-R,R]^m\big) \times\mathbb [-R,R]^n$, $R\ge 1$. 
However, this follows immediately by combining Proposition~\ref{prop3.1} and Lemma \ref{localizaiton2} with $s=\frac{m+n}{2}-\frac{m+n+2}{p}+\varepsilon$.   This proves \eqref{e1.2} .
The sharpness of the range $\alpha$ in \eqref{e1.2} is a special case of Proposition \ref{prop1.3}, which will 
proved in Section 5 below.
 The proof of Theorem~\ref{th1.1} is complete.
\end{proof}


 \medskip

\section{Proof of Theorem~\ref{th1.2}}\label{sec-4}

 To begin the proof of Theorem~\ref{th1.2}, let us first recall the setup of the $\ell^2$ decoupling theorem for the paraboloid     due to Bourgain and Demeter\cite{BoDe}, which plays an important role in our proof. 
 Let $\mathbb{P}^{n}=\big\{(\xi,|\xi|^2)\in \R^{n+1}:\,\xi\in [-1,1]^{n}\big\}$ be the truncated paraboloid in $\R^{n+1}$, and set
 $$
 \mathcal{N}_{\delta^2}(\mathbb{P}^{n})=
 \big\{(\xi,|\xi|^2+\mu)\in \mathbb{R}^{n+1}:\,\xi\in [-1,1]^{n},\,|\mu|\leq {\delta^2}/{2}\big\}.
 $$
 Let $\mathcal{P}_{\delta}^{n}$ be the cover of $\mathcal{N}_{\delta^2}(\mathbb{P}^{n})$ with curved regions
 $$
 \theta= \left\{(\xi,|\xi|^2+\mu)\in \mathbb R^{n+1}:\,\xi\in C_\theta,\,|\mu|\leq \delta^2\right\},
 $$
 where $C_{\theta}$ runs over all cubes $c+[-\frac{\delta}{2}, \frac{\delta}{2}]^{n}$ with 
 $c\in \frac{\delta}{2}\,\mathbb{Z}^{n}\cap [-1,1]^{n}$. For functions $f$ defined on $\mathbb R^{n+1}$, 
 denote by $f_\theta$ the Fourier restriction of $f$ to $\theta$.
 
The following (special) version of the $\ell^2$ decoupling theorem of Bourgain-Demeter \cite[Theorem 1.1]{BoDe} will be needed in the proof of Theorem \ref{th1.2}; see also \cite[Chapter 10]{De}.
 
 \begin{proposition} \label{BDv2}
 	Suppose $\text{supp} \hat{f}\subset \mathcal{N}_{\delta^2}(\mathbb{P}^{n})$. Then for $p\ge\frac{2(n+2)}{n}$ and any $\varepsilon>0$,
 	\begin{eqnarray}\label{eq3.5}
 		\|f\|_{L^p(\mathbb R^{n+1})}
 		\leq C_\varepsilon \delta^{\,-(\frac{n}{2}-\frac{n+2}{p}+\varepsilon)} \Big( \sum_{\theta\in \mathcal{P}^{n}_{\delta}} 
 		\|f_\theta\|_{L^p(\mathbb R^{n+1})}^2 \Big)^{1/2}.
 	\end{eqnarray}
 \end{proposition}

In order to apply Proposition~\ref{BDv2} to prove Theorem \ref{th1.2}, it will be convenient to 
introduce some notations from Fourier extension theory. 
From the   structure of $\Zm\times\Rn$ in the frequency, we consider the slices of $\mathbb{P}^{m+n}$, where the slices are $R^{-1}-$separated:
$$
\mathbb{P}^{m+n}_{R^{-1}}= \big\{(\tau,\xi,|\tau|^2+|\xi|^2)\in\mathbb{P}^{m+n}:\, \tau\in \big(R^{-1}\mathbb{Z}^m\big)\cap [-1,1]^m,\,\xi\in [-1,1]^n\big\}.
$$
For functions $g$ defined on $\mathbb{P}^{m+n}_{R^{-1}}$, define
\begin{align*}
	E_{\mathbb{P}^{m+n}_{R^{-1}}}g(x,y,t)
	= \sum_{\tau\in (R^{-1}\mathbb{Z}^m)\cap [-1,1]^m}\int_{[-1,1]^n} g(\tau,\xi)e^{i(\tau,\xi,|\tau|^2+|\xi|^2)\cdot(x,y,t)}\,d\xi,	
	\quad (x,y,t)\in\mathbb R^{m+n+1},
\end{align*}
  where we have identified (as we will often do below) $g(\tau,\xi,|\tau|^2+|\xi|^2)$ with $g(\tau,\xi)$.

In the following proposition, we first consider the $q=2$ case in Theorem \ref{th1.2}.  Recall that   $\{\sigma_K\}_{K\in\Z^{m+n}}\subseteq C_c^{\infty}({\mathbb R^{m+n}})$ is a smooth partition of unity, adapted to the translated unit cubes $Q_K=K+[-{1\over 2}, {1\over 2})^{m+n}$,
and  $\square_K$ is the operator given by $\widehat{\square_K f}= \sigma_K \hat f$,   $\langle K \rangle= (1+|K|^2)^{1/2}$. The  norm is defined by
$$
\|f\|_{M_{p,q}^\alpha(\TmRn)}= \bigg( \sum_{K\in\Z^{m+n}} \langle K \rangle^{q\alpha}\|\square_K f\|_{L^p(\TmRn)}^q \bigg)^{1/q}.
$$

\begin{proposition}\label{Propq2}
	Let $p\geq 2+4/(m+n)$. Then for any $\alpha>(m+n)(1/2-1/p)-2/p$, there exist a constant $C=C(m,n,p,\alpha)$ such that
\begin{align}\label{LF}
	\big\|e^{it\Delta_\TmRn}f\big\|_{L^p(\mathbb{T}^m\times \mathbb R^n\times [0,1])}\leq C R^\alpha \bigg( \sum_{K\in\Z^{m+n}}\| \square_K f\|_{L^p(\TmRn)}^2 \bigg)^{1/2}
\end{align}
for $f$ with $\text{supp\,}\hat{f}\subset \big(\Z^m\cap [-R,R]^m\big)\times [-R,R]^n$.
\end{proposition}

\begin{proof} 	Let $\hat{g}(k,\xi)=\hat{f}(Rk,R\xi)$ and $\hat{g}_K(k,\xi)=\widehat{\square_K f}(Rk,R\xi)$, then by a rescaling we have that
	\begin{align}\label{eq4.2}
		\big\|e^{it\Delta_{\mathbb T^m\times\mathbb R^n}} f\big\|_{L^p(\TmRn\times [0,1])}= R^{n-\frac{2m+n+2}{p}} \|E_{\mathbb{P}^{m+n}_{R^{-1}}}\hat{g}\|_{L^p([0,2\pi R^2]^m\times\Rn\times [0,R^2])}.
	\end{align}
	Let $\varphi$ be a Schwartz function on $\mathbb R$ with $\supp\hat{\varphi}\subset (-1/2,1/2)$ 
	such that $\varphi(t)\ge 1,\, t\in[-\pi,\pi]$. And let $\Phi(x)= \varphi(x_1)\varphi(x_2)\cdots\varphi(x_m)$. Then   the Fourier transform of
	$$
	F(x,y,t):= \Phi\bigg(\frac{x}{R^2}\bigg) \varphi\bigg(\frac{t}{R^2}\bigg)E_{\mathbb{P}^{m+n}_{R^{-1}}}\hat{g}(x,y,t)
	$$
	is supported in the $R^{-2}$ neighborhood of $\mathbb{P}^{m+n}_{R^{-1}}$, which is a subset of $ \mathcal{N}_{R^{-2}}(\mathbb{P}^{m+n})$. So we can apply Proposition~\ref{BDv2} to $F$ and obtain for any $\alpha>(m+n)(1/2-1/p)-2/p$,
	\begin{align}\label{eq4.3}
		\|E_{\mathbb{P}^{m+n}_{R^{-1}}}\hat{g}\|_{L^p([0,2\pi R^2]^m\times\Rn\times [0,R^2])}  \leq CR^\alpha \bigg( \sum_{K\in\Z^{m+n}}   \left\|\Phi\bigg(\frac{x}{R^2}\bigg) \varphi\bigg(\frac{t}{R^2}\bigg)E_{\mathbb{P}^{m+n}_{R^{-1}}} \widehat{g_K}\right\|_{L^p(\mathbb R^{m+n+1})}^2 \bigg)^{1/2}   
	\end{align}
	To calculate the $L^p$ norm on the right side, it suffices to consider $K=0$ the origin after a translation. But it is clear by \eqref{eq4.2} that
	\begin{align}\label{eq4.4}
		\bigg\|\Phi\bigg(\frac{x}{R^2}\bigg) \varphi\bigg(\frac{t}{R^2}\bigg)E_{\mathbb{P}^{m+n}_{R^{-1}}} \widehat{g_0}\bigg\|_{L^p(\mathbb R^{m+n+1})}    
&= R^{\frac{2m+n+2}{p}-n} \big\|\Phi(x)\varphi(t) e^{it\Delta_{\mathbb T^m\times\mathbb R^n}} (\square_0 f)\big\|_{L^p(\R^{m+n+1})}           \\
&\leq C R^{\frac{2m+n+2}{p}-n} \|\square_0 f\|_{L^p(\TmRn)}          \notag
	\end{align}
	since $\Phi$ and $\varphi$ are of rapid decay and $\widehat{\square_0}f$ is supported in the unit cube.  Clearly speaking, by \eqref{Mi} and \eqref{temp}, we have
	\begin{align*}
		\big\| e^{it\Delta_\TmRn} (\square_0 f)\big\|_{L^p(\TmRn)}\leq C(m,n,p) P(t) \| \square_0 f\|_{L^p(\TmRn)},
	\end{align*}
	where $P(t)$ is a polynomial. This makes sure the integral with respect to $t$ is convergent
	\begin{align*}
		\big\|\varphi(t) e^{it\Delta_{\mathbb T^m\times\mathbb R^n}} (\square_0 f)\big\|_{L^p(\TmRn\times\R)}
		&= \bigg(\int_\R \varphi(t)^p \big\| e^{it\Delta_{\mathbb T^m\times\mathbb R^n}} (\square_0 f)\big\|_{L^p(\TmRn)}^p  dt\bigg)^{1/p}   \\
		&\leq \| \square_0 f\|_{L^p(\TmRn)}.
	\end{align*}
In conclusion, combining all the estimates  \eqref{eq4.2}, \eqref{eq4.3} and \eqref{eq4.4}, we finally obtain 
\begin{align}\label{LF}
	\big\|e^{it\Delta_\TmRn}f\big\|_{L^p(\mathbb{T}^m\times \mathbb R^n\times [0,1])}\leq C(m,n,p,\alpha) R^\alpha \bigg( \sum_{K\in\Z^{m+n}} \| \square_K f\|_{L^p(\TmRn)}^2 \bigg)^{1/2}
\end{align}
for $f$ with $\text{supp\,}\hat{f}\subset \big(\Z^m\cap [-R,R]^m\big)\times [-R,R]^n$. This completes the proof of Proposition~\ref{Propq2}.
\end{proof}

\begin{proof}[Proof of Theorem \ref{th1.2}]
	First we show (a), (b) and (c) in Theorem \ref{th1.2} for $q=2$.
	Recall that  $P_j$ is the Littlewood-Paley operator on $\TmRn$  in Proposition \ref{AC2v2}. From the Littlewood-Paley  estimate, Minkowski's inequality and Proposition \ref{Propq2}, we obtain		
		\begin{eqnarray}\label{m2}
			\|e^{it\Delta_\TmRn}f\|_{L^p(\mathbb{T}^m\times \mathbb R^n\times [0,1])}
			&\leq& C \bigg\|\big(\sum_{j=0}^{\infty} |e^{it\Delta_\TmRn}P_jf|^2\big)^{1/2} \bigg\|_{L^p(\mathbb{T}^m\times \mathbb R^n\times [0,1])}      \nonumber   \\
			&\leq& C\bigg(\sum_{j=0}^{\infty} \|e^{it\Delta_\TmRn}P_jf\|_{L^p(\mathbb{T}^m\times \mathbb R^n\times [0,1])} ^2\bigg)^{1/2}       \nonumber   \\
			&\leq& C(m,n,p,\alpha) \bigg( \sum_K \langle K\rangle^{2\alpha}\| \square_K f\|_{L^p(\TmRn)}^2 \bigg)^{1/2}  \nonumber\\
			&=& C(m,n,p,\alpha) \|f\|_{M^\alpha_{p,2}(\TmRn)},
		\end{eqnarray}
		where $\alpha>0$ for $2\leq p \leq 2+ 4/(m+n)$; and $\alpha>(m+n) (1/2-1/p) -2/p +\epsilon$ for $ 2+ 4/(m+n)<p\leq \infty$.
		
		Next we consider \eqref{LSM} for $q=1$. It  is clear  that for any $p\geq 1$,  
		\begin{eqnarray}\label{e11}
			\|e^{it\Delta_{\mathbb T^m\times\mathbb R^n}}f\|_{L^p(\TmRn\times [0,1])}
			 &\leq& C\sum_K \|e^{it\Delta_{\mathbb T^m\times\mathbb R^n}}\square_K f\|_{L^p(\TmRn\times [0,1])} \nonumber \\         
			&\leq& C\sum_K \|\square_K f\|_{L^p(\TmRn)}.      
		\end{eqnarray}
	For $1\leq q\leq 2$, (a) follows from  \eqref{m2} and interpolating with \eqref{e11},  and for $q\geq 2$ we use the embedding 
	 \eqref{e1.16}. (b) follows for $q\geq 2$ via \eqref{e1.16}. (c) follows from interpolating (b) for $q=2$ with \eqref{e11}.  
	The proof of Theorem~\ref{th1.2} is complete.		
\end{proof}



\section{Sharpness: Proof of Proposition~\ref{prop1.3}}\label{sec-nec}
\setcounter{equation}{0}

The proof of part (i) of Proposition~\ref{prop1.3} is an adaptation of the corresponding result for $\mathbb R^n$ by Rogers \cite{Ro}.
Let $\phi$ be a Schwartz function on $\mathbb R$ with $\text{supp}\hat\phi\subset (-1,1)$, such that 
$\hat\phi(\xi)\ge 0,\,\xi\in (-1,1)$, and such that $\hat\phi(\xi)\ge 1,\,\xi\in (-1/2,1/2)$. For $h\in(0,1/2)$, let
$$\phi_h(y)=\frac{1}{\sqrt h}\phi\Big(\frac yh\Big),
\quad y\in\mathbb R, $$
and let
$$\widetilde\phi_h(x)
=\sqrt h\sum_{k\in\mathbb Z} \hat\phi(hk)e^{ikx},
\quad x\in\mathbb T.$$
Note that $\phi_h$ and $\widetilde\phi_h$ are normalized in $L^2$.
For $m, n\ge 1$, let
$$\Phi_h(y)=\phi_h(y_1)\cdots \phi_h(y_n),\quad y=(y_1,\cdots,y_n)\in\mathbb R^n,$$
and let
$$\widetilde\Phi_h(x)=\widetilde\phi_h(x_1)\cdots\widetilde\phi_h(x_m),\quad
x=(x_1,\cdots,x_m)\in\mathbb T^m.$$
By the semiclassical dispersion estimate of  Burq, G\'erard and Tzvetkov \cite[Lemma 2.5]{BuGeTz}, we have
$$\big\|e^{i \varepsilon_0 h\Delta_{\mathbb T}}\widetilde\phi_h\big\|_{L^\infty(\mathbb T)}\le C$$
for a sufficiently small constant $\varepsilon_0$. Hence
$$\big\|e^{i \varepsilon_0 h\Delta_{\mathbb T^m}}\widetilde\Phi_h\big\|_{L^\infty(\mathbb T^m)}
=\big\|e^{i \varepsilon_0 h\Delta_{\mathbb T}}\widetilde\phi_h\big\|_{L^\infty(\mathbb T)}^m\le C.$$
Since $e^{i t\Delta_{\mathbb T^m}}$ is $2\pi$-periodic in $t$, it follows that
\begin{equation}\label{disperse1}
\big\|e^{-i (2\pi-\varepsilon_0 h)\Delta_{\mathbb T^m}}\widetilde\Phi_h\big\|_{L^\infty(\mathbb T^m)}\le C.
\end{equation}
On the other hand, since $2\pi-\varepsilon_0 h\approx 1$, by standard dispersion estimate (cf. \cite{Ro}, Section 2),
we have
\begin{equation}\label{disperse2}
\big\|e^{-i (2\pi-\varepsilon_0 h)\Delta_{\mathbb R^n}}\Phi_h\big\|_{L^\infty(\mathbb R^n)}
=\big\|e^{-i (2\pi-\varepsilon_0h)\Delta_{\mathbb R}}\phi_h\big\|_{L^\infty(\mathbb R)}^n
\le C h^{\frac n2}.
\end{equation}
Therefore, setting
$$f(x,y)= e^{-i(2\pi-\varepsilon_0 h) \Delta_{\mathbb{T}^m}}\widetilde\Phi_h(x)\cdot e^{-i(2\pi-\varepsilon_0 h) \Delta_{\R^n}}\Phi_h(y),$$
it follows from Corollary \ref{cor}(i), \eqref{disperse1}, and \eqref{disperse2} that
\begin{align}\label{est4}
	\|f\|_{W^{\alpha,p}(\mathbb{T}^m\times \R^n)}\leq C h^{-\alpha}h^{n(\frac12-\frac1p)}.
\end{align}

Notice that for $|\tau|\le \varepsilon_0 h^2$ and $|x|, |y|\le \varepsilon_0 h$ (decreasing $\varepsilon_0$ if necessary), by Fourier inversion, we have
$$
	\big|e^{i\tau  \Delta_{\mathbb{T}^m}}\widetilde\Phi_h(x)\big|\ge C h^{-\frac m2},\quad
\big|e^{i\tau \Delta_{\R^n}}\Phi_h(y)\big|\geq C h^{-\frac n2}.
$$
Thus, for $t\in 2\pi-\varepsilon_0 h+[0,\varepsilon_0 h^2]$ and $|x|, |y|\le \varepsilon_0 h$,
$$
	\left|e^{it \Delta_{\mathbb T^m\times\mathbb R^n}} f(x,y)\right|=
	\big|e^{i(t-1+\varepsilon_0 h)  \Delta_{\mathbb{T}^m}}\widetilde\Phi_h(x)\big|\cdot
\big|e^{i(t-1+\varepsilon_0  h) \Delta_{\R^n}}\Phi_h(y)\big|
	\geq C h^{-\frac {m+n}{2}}.
$$
After integration, this implies
\begin{align}\label{est5}
	\big\| e^{it\Delta_{\mathbb T^m\times\mathbb R^n}}f \big\|_{L^q(\mathbb{T}^m\times \R^n, L^r[0,1])}\geq C h^{-(m+n)(\frac12-\frac1q)} h^{\frac 2r}.
\end{align}
Combining (\ref{est4}) and (\ref{est5}), and letting $h\to 0^+$, we see that for  (\ref{eq:LS}) to hold, one must have
$$
\alpha
\ge
(m+n)\Big(\frac12-\frac1q\Big)
+n\Big(\frac 12-\frac1p\Big)
-\frac2r.
$$

Next we prove part (ii) of Proposition~\ref{prop1.3}.	Let $\widetilde\Phi_h, \Phi_h$ be as above. Set $f(x,y)=\widetilde\Phi_h(x)\Phi_h(y)$. Since $|\square_K f|\approx h^{(m+n)/2}|\widehat{\sigma_K}|$ for $|K|\leq h^{-1}$, we have
	\begin{align}\label{eq5.7}
		\|f\|_{M^\alpha_{p,q}}
		&\leq C\bigg(\sum_{|K|\leq h^{-1}} h^{-\alpha q}h^{\frac{(m+n)q}{2}}\bigg)^{1/q}    \\
		&= C h^{(m+n)(\frac12-\frac1q)-\alpha}.      \notag
	\end{align}
	On the other hand, for $0\leq t\leq \varepsilon_0 h^2$ and $|x|,|y|\leq \varepsilon_0 h$, 
	$$
	\left|e^{it \Delta_{\mathbb T^m\times\mathbb R^n}} f(x,y)\right|
	=\big|e^{it  \Delta_{\mathbb{T}^m}}\widetilde\Phi_h(x)\big|\cdot
	\big|e^{it \Delta_{\R^n}}\Phi_h(y)\big|
	\geq C h^{-\frac {m+n}{2}}. 
	$$
	Thus
	\begin{align}\label{eq5.8}
		\big\| e^{it\Delta_{\mathbb T^m\times\mathbb R^n}}f \big\|_{L^p(\mathbb{T}^m\times \R^n\times [0,1])}\geq C h^{-(m+n)(\frac12-\frac1p)+\frac2p}.
	\end{align}
	Combining \eqref{eq5.7} and \eqref{eq5.8}, and letting $h\to 0^+$, we obtain the necessary condition
	$$
	\alpha\geq (m+n)\Big(1-\frac 1p-\frac 1q\Big)-\frac2/p.
	$$
	For the condition $\alpha\ge 0$, choose $K_0$ such that $h^{-1}/2\leq |K_0|\leq h^{-1}$, and set $\hat{f}= \sigma_{K_0}$. Since $\|f\|_{M^\alpha_{p,q}}\leq C h^{-\alpha}$,\;  $\big\| e^{it\Delta_{\mathbb T^m\times\mathbb R^n}}f \big\|_{L^p(\mathbb{T}^m\times \R^n\times [0,1])}\geq C$,  \eqref{eq:modulationn} immediately implies $\alpha\geq 0$.
	Summarizing, for \eqref{eq:modulationn} to hold, one must have
	$$
	\alpha\geq \max\left\{0,(m+n)\bigg(1-\frac1p-\frac1q\bigg)-\frac2p\right\}.
	$$
 This proves part (ii) of  Proposition~\ref{prop1.3}, and the proof of  Proposition~\ref{prop1.3} is complete.

\begin{remark}	 
 By taking $p=q$ and $r=\infty$ in the proof of part (i) above, one sees that the 
 fixed-time estimate (\ref{fixtime}) cannot hold uniformly near $t=2\pi$ when $\alpha<(m+2n)(\frac12-\frac1p)$. On the other hand, it can be shown that when $t$ is a rational multiple of $2\pi$, the range in \eqref{fixtime} extends to $\alpha\ge 2n(\frac12-\frac1p)$ (cf. \cite{T}).

\end{remark}

 \noindent
{\bf Acknowledgments.} 
 The authors would like to thank S.M. Guo for helpful discussions.
X. Chen was supported by the NNSF of China, Grant Nos. 11901593 and 12071490. Z. Guo is supported by ARC DP200101065.
L. Yan was supported by  
National key R$\&$D program of China: No. 2022YFA1005700, the NNSF of China
11871480  and by the Australian Research Council (ARC) through the research
grant DP190100970.


\bigskip

\bibliographystyle{amsplain}

\end{document}